\documentclass{article}
\usepackage{amssymb}
\usepackage{amsfonts}
\usepackage{amsmath}
\usepackage{harvard}

\newtheorem{theorem}{Theorem}

\newtheorem{corollary}[theorem]{Corollary}

\newtheorem{example}[theorem]{Example}

\newtheorem{lemma}[theorem]{Lemma}

\newenvironment{proof}[1][Proof]{\noindent\textbf{#1.} }{\ \rule{0.5em}{0.5em}}

\input{tcilatex}

\begin{document}

\title{Integer Powers of Complex Tridiagonal and Anti-Tridiagonal Matrices}
\author{ Hatice K\"{u}bra Duru\thanks{%
hkduru@selcuk.edu.tr} \&Durmu\c{s} Bozkurt\thanks{%
dbozkurt@selcuk.edu.tr} \\
Department of Mathematics, Science Faculty of Sel\c{c}uk University}
\maketitle

\begin{abstract}
In this paper, we derive the general expression of the $r-$th power for some 
$n$-square complex tridiagonal matrices.Also one type is given eigenvalues
and eigenvectors of complex anti-tridiagonal matrices Additionally, we
obtain the complex factorizations of Fibonacci polynomials.
\end{abstract}

\section{Introduction}

\qquad Elouafi and Hadj [1] offered tridiagonal matrix powers and inverse.
Guti\'{e}rrez [2,3] obtained a general expression for the entries of the $q-$%
th power $\left( q\in 
\mathbb{N}
\right) $ of the $n\times n$ complex tridiagonal matrix $tridiag_{n}\left(
a_{1},a_{0},a_{-1}\right) ~$for all $n\in 
\mathbb{N}
~$and $2\left( q-1\right) \leq n$. Rimas [4-8] enquired the arbitrary
positive integer powers for some tridiagonal matrices. \"{O}tele\c{s} and
Akbulak [9,10] generalized Rimas's the some results and get complex
factorization formula for the generalized Fibonacci-Pell numbers.

Let%
\begin{equation}
A:=\left[ 
\begin{array}{cccccc}
a & 2b &  &  &  & 0 \\ 
b & a & b &  &  &  \\ 
& b & a & b &  &  \\ 
&  & \ddots & \ddots & \ddots &  \\ 
&  &  & b & a & 2b \\ 
0 &  &  &  & b & a%
\end{array}%
\right]  \label{1}
\end{equation}%
and%
\begin{equation}
A^{\dag }:=\left[ 
\begin{array}{cccccc}
a & b &  &  &  & 0 \\ 
b & a & -b &  &  &  \\ 
& -b & a & b &  &  \\ 
&  & b & a & -b &  \\ 
&  &  & -b & a & \ddots \\ 
0 &  &  &  & \ddots & \ddots%
\end{array}%
\right]  \label{2}
\end{equation}%
where $b\neq 0$ and $a,b\in 
\mathbb{C}
$. In this paper, we want to $r-$th power obtain of an $n-$square complex
tridiagonal matrices in $(1)$ and $(2)$.

We derive expression of the $r-$th power $\left( r\in 
\mathbb{N}
\right) $ a matrix applying the well-known expression $%
G^{r}=SJ^{r}S^{-1}[13] $, where $J$ \ is the Jordan's form of the matrix $G$
and $S~$ is the transforming matrix of $G$. We need the eigenvalues and
eigenvectors of the matrices $A$ and $A^{\dag }$, respectively,\ to
calculate transforming matrices.

\bigskip Let $Q$ be the following $n\times n$ tridiagonal matrix%
\begin{equation}
Q:=\left[ 
\begin{array}{cccccc}
0 & 2 &  &  &  &  \\ 
1 & 0 & 1 &  &  &  \\ 
& 1 & 0 & 1 &  &  \\ 
&  & \ddots & \ddots & \ddots &  \\ 
&  &  & 1 & 0 & 2 \\ 
&  &  &  & 1 & 0%
\end{array}%
\right] .\   \label{3}
\end{equation}%
Then, the eigenvalues of $Q$ is%
\begin{equation*}
\mu _{k}=2\cos \left( \frac{(k-1)\pi }{n-1}\right) ,\ k=1,2,\ldots ,n\ [4].
\end{equation*}

The Chebyshev polynomials of the first kind $T_{n}(x)$ and second kind $%
U_{n}(x)$ are defined as $[14]$%
\begin{equation}
T_{n}(x)=\cos \left( n\arccos x\right) ,~~~-1\leq x\leq 1  \label{4}
\end{equation}%
and%
\begin{equation}
U_{n}(x)=\frac{\sin ((n+1)\arccos x)}{\sin (\arccos x)},\ \ -1\leq x\leq 1.
\label{5}
\end{equation}

All roots of the polynomial $U_{n}(x)$ are included in the interval $[-1,1]$
and can be found using the relation%
\begin{equation}
x_{nk}=\cos \left( \frac{k\pi }{n+1}\right) ,\ \ k=1,2,\ldots ,n.  \label{6}
\end{equation}

\section{Eigenvalues and eigenvectors of $A$ and $A^{\dagger }$}

\begin{theorem}
\bigskip Let $A$ be as in $(1)$. Then the eigenvalues and eigenvectors of
the matrix A are 
\begin{equation}
\lambda _{k}=a+2b\cos \left( \frac{(k-1)\pi }{n-1}\right) ,\ k=1,2,\ldots ,n
\label{7}
\end{equation}%
and%
\begin{equation}
x_{kj}=\left\{ 
\begin{array}{l}
~~T_{k-1}\left( \frac{\delta _{j}}{2}\right) ;k=1,2,\ldots ,n-1 \\ 
\frac{1}{2}T_{k-1}\left( \frac{\delta _{j}}{2}\right) ;k=n%
\end{array}%
\right. j=1,2,\ldots ,n  \label{8}
\end{equation}%
where $\delta _{j}=\frac{\lambda _{j}-a}{b}$ and $T_{k}(x)$ is Chebyshev
polynomial of the first kind.
\end{theorem}

\begin{proof}
\bigskip Let $B$ be as in the following $n\times n$ tridiagonal matrix:%
\begin{equation}
B:=\left[ 
\begin{array}{cccccc}
\frac{a}{b} & 2 &  &  &  &  \\ 
1 & \frac{a}{b} & 1 &  &  &  \\ 
& 1 & \frac{a}{b} & 1 &  &  \\ 
&  & \ddots & \ddots & \ddots &  \\ 
&  &  & 1 & \frac{a}{b} & 2 \\ 
&  &  &  & 1 & \frac{a}{b}%
\end{array}%
\right]  \label{9}
\end{equation}%
Then, the characteristic polynomials of $B$ are%
\begin{equation}
D_{n}(\alpha )=(\alpha ^{2}-4)P_{n-2}(\alpha )  \label{10}
\end{equation}%
where $\alpha =\lambda -\frac{a}{b}$ and%
\begin{equation}
P_{n}(\alpha )=\alpha P_{n-1}(\alpha )-P_{n-2}(\alpha )  \label{11}
\end{equation}%
with initial conditions $P_{0}(\alpha )=1,~P_{1}(\alpha )=\alpha
,~P_{2}(\alpha )=\alpha ^{2}-1.$

Solution of difference equation in $(11)$ is $P_{n}(\alpha )=U_{n}(\frac{%
\alpha }{2})\ $here $U_{n}(x)$ is Chebyshev polynomial of the second kind.
So the equality in $\left( 10\right) $ written as%
\begin{equation*}
D_{n}(\alpha )=(\alpha ^{2}-4)U_{n-2}(\tfrac{\alpha }{2}).
\end{equation*}%
Then, we have%
\begin{equation*}
\alpha _{k}=2\cos \left( \frac{(k-1)\pi }{n-1}\right) ,\ k=1,2,\ldots ,n.
\end{equation*}%
So, the eigenvalues of $A$ are%
\begin{equation*}
\lambda _{k}=a+2b\cos \left( \frac{(k-1)\pi }{n-1}\right) ,~~for~~\
k=1,2,\ldots ,n~.
\end{equation*}%
Components eigenvectors of the matrix $A$ are the solutions of the following
homogeneous linear equations system:%
\begin{equation}
(\lambda _{k}I_{n}-A)x=0  \label{12}
\end{equation}%
where $\lambda _{k}$ is the $k-$th eigenvalue of the matrix $A$ ($%
k=1,2,\ldots ,n)$. The following equations system $(12)$ written, we possess%
\begin{equation}
\left. 
\begin{array}{r}
(\lambda _{k}-a)x_{1}-2bx_{2}=0 \\ 
-bx_{1}+(\lambda _{k}-a)x_{2}-bx_{3}=0 \\ 
-bx_{2}+(\lambda _{k}-a)x_{3}-bx_{4}=0 \\ 
\vdots \ \ \ \  \\ 
-bx_{n-2}+(\lambda _{k}-a)x_{n-1}-2bx_{n}=0 \\ 
-bx_{n-1}+(\lambda _{k}-a)x_{n}=0%
\end{array}%
\right\}  \label{13}
\end{equation}%
Dividing all terms of the each equation in system $(13)$ by $b\neq 0$,
substituting $\delta _{j}=\frac{\lambda _{j}-a}{b}\left( j=1,2,\ldots
,n\right) .~$Since rank of the system is $n-1$; choosing $x_{1}=1$ and
solving the set of the system $(13)$ as regards \ $x_{1}$, 
\begin{equation*}
x_{kj}=\left\{ 
\begin{array}{l}
~~~T_{k-1}\left( \frac{\delta _{j}}{2}\right) ;~k=1,2,\ldots ,n-1 \\ 
~\frac{1}{2}T_{k-1}\left( \frac{\delta _{j}}{2}\right) ;~k=n%
\end{array}%
\right. j=1,2,\ldots ,n
\end{equation*}%
where $\delta _{j}=\frac{\lambda _{j}-a}{b}$ and $T_{k}(x)$ is Chebyshev
polynomial of the first kind.
\end{proof}

\begin{theorem}
\bigskip Let $A^{\dagger }$ be as in $(2)$. Then the eigenvalues and
eigenvectors of the matrix $A^{\dagger }$ are 
\begin{equation}
\lambda _{k}^{\dagger }=a-2b\cos \left( \frac{k\pi }{n+1}\right) ,\
k=1,2,\ldots ,n  \label{14}
\end{equation}%
and%
\begin{equation}
y_{kj}=r_{k-1}U_{k-1}\left( \frac{\psi _{j}}{2}\right) ;j,k=1,2,\ldots ,n
\label{15}
\end{equation}%
here $\psi _{j}=\frac{\lambda _{j}^{\dag }-a}{b},r_{k-1}=\left\{ 
\begin{array}{c}
~~1,k-1\equiv 0~or~1~(\func{mod}4) \\ 
-1,k-1\equiv 2~or~3~(\func{mod}4)%
\end{array}%
\right. $ and $U_{k}(x)$ is Chebyshev polynomial of the second kind.

\begin{proof}
Let%
\begin{equation}
B^{\dag }:=\left[ 
\begin{array}{cccccc}
\frac{a}{b} & 1 &  &  &  & 0 \\ 
1 & \frac{a}{b} & -1 &  &  &  \\ 
& -1 & \frac{a}{b} & 1 &  &  \\ 
&  & 1 & \frac{a}{b} & -1 &  \\ 
&  &  & -1 & \frac{a}{b} & \ddots \\ 
0 &  &  &  & \ddots & \ddots%
\end{array}%
\right] .  \label{16}
\end{equation}%
Let $Q^{\dag }$ be the following $n\times n$ tridiagonal matrix%
\begin{equation}
Q^{\dagger }:=\left[ 
\begin{array}{cccccc}
0 & 1 &  &  &  &  \\ 
1 & 0 & -1 &  &  &  \\ 
& -1 & 0 & 1 &  &  \\ 
&  & 1 & 0 & -1 &  \\ 
&  &  & -1 & 0 & \ddots \\ 
&  &  &  & \ddots & \ddots%
\end{array}%
\right] .  \label{17}
\end{equation}%
The eigenvalues of $Q^{\dag }$ are defined by the roots of the
characteristic equation%
\begin{equation*}
\left\vert Q^{\dag }-\theta I\right\vert =0.
\end{equation*}%
Let%
\begin{equation}
D_{n}^{\dag }\left( \theta \right) =\left\vert 
\begin{array}{cccccc}
\theta & 1 &  &  &  &  \\ 
1 & \theta & -1 &  &  &  \\ 
& -1 & \theta & 1 &  &  \\ 
&  & 1 & \theta & -1 &  \\ 
&  &  & -1 & \theta & \ddots \\ 
&  &  &  & \ddots & \ddots%
\end{array}%
\right\vert  \label{18}
\end{equation}%
and%
\begin{equation}
P_{n}(\theta )=\left\vert 
\begin{array}{cccccc}
\theta & 1 &  &  &  &  \\ 
1 & \theta & 1 &  &  &  \\ 
& 1 & \theta & 1 &  &  \\ 
&  & 1 & \theta & 1 &  \\ 
&  &  & 1 & \theta & \ddots \\ 
&  &  &  & \ddots & \ddots%
\end{array}%
\right\vert  \label{19}
\end{equation}%
here $\theta \in 
\mathbb{R}
$. Then%
\begin{equation*}
\left\vert Q^{\dag }-\theta I\right\vert =D_{n}^{\dag }\left( \theta \right)
\end{equation*}%
and%
\begin{equation}
D_{n}^{\dag }\left( \theta \right) =P_{n}(\theta ).  \label{20}
\end{equation}%
Let us prove by $\left( 20\right) $ the inductive method. For the basis
step, we possess 
\begin{eqnarray*}
D_{1}^{\dag }\left( \theta \right) &=&\theta =P_{1}(\theta ) \\
D_{2}^{\dag }\left( \theta \right) &=&\theta ^{2}-1=P_{2}(\theta ) \\
D_{3}^{\dag }\left( \theta \right) &=&\theta ^{3}-2\theta =P_{3}(\theta ).
\end{eqnarray*}%
We suppose $D_{n-1}^{\dag }\left( \theta \right) =\varkappa =P_{n-1}(\theta
) $ and $D_{n}^{\dag }\left( \theta \right) =\varrho =P_{n}(\theta )$ for $%
n\geq 3$. The well known recurrence relations is 
\begin{equation}
\left\vert H\left( n\right) \right\vert =h_{n,n}\left\vert H\left(
n-1\right) \right\vert -h_{n-1,n}h_{n,n-1}\left\vert H\left( n-2\right)
\right\vert ~\left[ 15\right] .  \label{21}
\end{equation}%
From $\left( 21\right) $, if $n$ is a positive odd integer, since $%
h_{n,n}=\theta ,~h_{n-1,n}=-1,~h_{n,n-1}=-1,~D_{n-1}^{\dag }\left( \theta
\right) =\varkappa $ and $D_{n}^{\dag }\left( \theta \right) =\varrho ,$ 
\begin{equation*}
D_{n+1}^{\dag }\left( \theta \right) =\theta \varrho -\left( -1\right)
\left( -1\right) \varkappa =\theta \varrho -\varkappa ,
\end{equation*}%
if $n$ is a positive even integer, since $h_{n,n}=\theta ,~h_{n-1,n}=1$ and $%
h_{n,n-1}=1,$ 
\begin{equation*}
D_{n+1}^{\dag }\left( \theta \right) =\theta \varrho -\varkappa
\end{equation*}%
and since $h_{n,n}=\theta ,~h_{n-1,n}=1,~h_{n,n-1}=1,~P_{n-1}(\theta
)=\varkappa $ and $P_{n}(\theta )=\varrho ,$%
\begin{equation*}
P_{n+1}(\theta )=\theta \varrho -\varkappa .
\end{equation*}%
For $\forall n\in 
\mathbb{Z}
,$ we have$~D_{n}^{\dag }\left( \theta \right) =P_{n}(\theta ).~$From $%
\left( 11\right) $, we obtain%
\begin{equation*}
D_{n}^{\dag }\left( \theta \right) =U_{n}\left( \frac{\theta }{2}\right) .
\end{equation*}%
Then the eigenvalues of the matrix $Q^{\dag }$ are%
\begin{equation*}
\theta _{k}=-2\cos \left( \frac{k\pi }{n+1}\right) ,\ k=1,2,\ldots ,n.
\end{equation*}%
The proof can be done easily for the matrix $A^{\dag }$ similar to Theorem 1.
\end{proof}
\end{theorem}

\section{\protect\bigskip \textbf{The integer powers of the matrices }$A$%
\textbf{\ and }$A^{\dagger }$}

Consider the relations $A=KJK^{-1}$ and $A^{\dag }=TJ^{\dag }T^{-1},$where $%
J $ and $J^{\dag }$ are the Jordan$^{\prime }$s forms of the matrices $A$
and $A^{\dag },~K$ and $T$\ are transforming matrices of the matrix $A$ and $%
A^{\dag }$, respectively. Since all the eigenvalues of $A$ and $A^{\dag }$
are simple, columns of the transforming matrices $K$ and $T$ are the
eigenvectors of the matrices $A$ and $A^{\dag },$ respectively $[13]$. Also
all\ eigenvalues$\ \lambda _{k}$ and $\lambda _{k}^{\dagger }$ corresponds
single Jordan cells $J_{i}(\ \lambda _{k})$ and $J_{i}^{\dag }(\ \lambda
_{k}^{\dag })$\ in the matrix $J$ and $J^{\dag },~$respectively$.$ Then, we
write down the Jordan$^{\prime }$s forms of the matrices $A$ and $A^{\dag }$%
\begin{equation}
J=diag(\lambda _{1},\lambda _{2},\lambda _{3},...,\lambda _{n})  \label{22}
\end{equation}%
and%
\begin{equation}
J^{\dag }=diag(\lambda _{1}^{\dag },\lambda _{2}^{\dag },\lambda _{3}^{\dag
},...,\lambda _{n}^{\dag }).  \label{23}
\end{equation}

From $\left( 8\right) $ and $\left( 15\right) $, we can write the
transforming matrices $K$ and $T$ \ as%
\begin{equation}
K=[x_{kj}]=\left\{ 
\begin{array}{l}
~~~T_{k-1}\left( \frac{\delta _{j}}{2}\right) ;~k=1,2,\ldots ,n-1 \\ 
\frac{1}{2}T_{k-1}\left( \frac{\delta _{j}}{2}\right) ~;~k=n%
\end{array}%
\right. j=1,2,\ldots ,n  \label{24}
\end{equation}%
and%
\begin{equation}
T=\left[ y_{kj}\right] =r_{k-1}U_{k-1}\left( \frac{\psi _{j}}{2}\right)
\label{25}
\end{equation}%
where $\delta _{j}=\frac{\lambda _{j}-a}{b},~$ $\psi _{j}=\frac{\lambda
_{j}^{\dag }-a}{b}$ and%
\begin{equation*}
r_{k-1}=\left\{ 
\begin{array}{c}
~~1,k-1\equiv 0~or~1~(\func{mod}4) \\ 
-1,k-1\equiv 2~or~3~(\func{mod}4)%
\end{array}%
\right. .
\end{equation*}%
Considering $\left( 24\right) $ and $\left( 25\right) $, we write down the
transforming matrices $K$ and $T$, respectively,%
\begin{equation}
K=\left[ 
\begin{array}{ccccc}
T_{0}\left( \frac{\delta _{1}}{2}\right) & T_{0}\left( \frac{\delta _{2}}{2}%
\right) & \cdots & T_{0}\left( \frac{\delta _{n-1}}{2}\right) & T_{0}\left( 
\frac{\delta _{n}}{2}\right) \\ 
T_{1}\left( \frac{\delta _{1}}{2}\right) & T_{1}\left( \frac{\delta _{2}}{2}%
\right) & \cdots & T_{1}\left( \frac{\delta _{n-1}}{2}\right) & T_{1}\left( 
\frac{\delta _{n}}{2}\right) \\ 
\vdots & \vdots & \ddots & \vdots & \vdots \\ 
T_{n-2}\left( \frac{\delta _{1}}{2}\right) & T_{n-2}\left( \frac{\delta _{2}%
}{2}\right) & \cdots & T_{n-2}\left( \frac{\delta _{n-1}}{2}\right) & 
T_{n-2}\left( \frac{\delta _{n}}{2}\right) \\ 
\frac{1}{2}T_{n-1}\left( \frac{\delta _{1}}{2}\right) & \frac{1}{2}%
T_{n-1}\left( \frac{\delta _{2}}{2}\right) & \cdots & \frac{1}{2}%
T_{n-1}\left( \frac{\delta _{n-1}}{2}\right) & \frac{1}{2}T_{n-1}\left( 
\frac{\delta _{n}}{2}\right)%
\end{array}%
\right]  \label{26}
\end{equation}%
and%
\begin{equation}
T=\left[ 
\begin{array}{cccc}
r_{0}U_{0}\left( \frac{\psi _{1}}{2}\right) & r_{0}U_{0}\left( \frac{\psi
_{2}}{2}\right) & \cdots & r_{0}U_{0}\left( \frac{\psi _{n}}{2}\right) \\ 
r_{1}U_{1}\left( \frac{\psi _{1}}{2}\right) & r_{1}U_{1}\left( \frac{\psi
_{2}}{2}\right) & \cdots & r_{1}U_{1}\left( \frac{\psi _{n}}{2}\right) \\ 
\vdots & \vdots & \ddots & \vdots \\ 
r_{n-2}U_{n-2}\left( \frac{\psi _{1}}{2}\right) & r_{n-2}U_{n-2}\left( \frac{%
\psi _{2}}{2}\right) & \cdots & r_{n-2}U_{n-2}\left( \frac{\psi _{n}}{2}%
\right) \\ 
r_{n-1}U_{n-1}\left( \frac{\psi _{1}}{2}\right) & r_{n-1}U_{n-1}\left( \frac{%
\psi _{2}}{2}\right) & \cdots & r_{n-1}U_{n-1}\left( \frac{\psi _{n}}{2}%
\right)%
\end{array}%
\right] .  \label{27}
\end{equation}

Denoting $~j$th column of the matrix $K^{-1}$ by $\tau _{j}$ and
implementing the necessary transformations, we have 
\begin{equation*}
\tau _{j}=\gamma _{j}\left[ 
\begin{array}{c}
\beta _{1}T_{j-1}(\frac{\delta _{1}}{2}) \\ 
\beta _{2}T_{j-1}(\frac{\delta _{2}}{2}) \\ 
\beta _{3}T_{j-1}(\frac{\delta _{3}}{2}) \\ 
\vdots \\ 
\beta _{n-1}T_{j-1}(\frac{\delta _{n-1}}{2}) \\ 
\beta _{n}T_{j-1}(\frac{\delta _{n}}{2})%
\end{array}%
\right] ,\ j=\overline{1,n}
\end{equation*}%
where $\gamma _{j}=\left\{ 
\begin{array}{l}
1,~j=1 \\ 
2,1<j\leq n%
\end{array}%
\right. ~$and~ $\beta _{k}=\frac{1}{2n-2}\left\{ 
\begin{array}{l}
1,~k=1,n \\ 
2,1<k<n%
\end{array}%
\right. .$

Let

\begin{equation*}
A^{s}=KJ^{s}K^{-1}=U(s)=(u_{ij}(s))
\end{equation*}%
here%
\begin{equation*}
s=\left\{ 
\begin{array}{l}
s\in 
\mathbb{N}
,~~n~~odd \\ 
s\in 
\mathbb{Z}
,~~n~~even.%
\end{array}%
\right.
\end{equation*}%
Thus

\begin{equation}
u_{ij}(s)=\gamma _{j}\sum_{k=1}^{n}\lambda _{k}^{s}\beta _{k}T_{i-1}\left( 
\frac{\delta _{k}}{2}\right) T_{j-1}\left( \frac{\delta _{k}}{2}\right)
\label{28}
\end{equation}%
where $i=1,2,\ldots ,n-1$; $j=1,2,\ldots ,n$ and

\begin{equation}
u_{ij}(s)=\frac{\gamma _{j}}{2}\sum_{k=1}^{n}\lambda _{k}^{s}\beta
_{k}T_{i-1}\left( \frac{\delta _{k}}{2}\right) T_{j-1}\left( \frac{\delta
_{k}}{2}\right)  \label{29}
\end{equation}%
for$~i=n$; $j=1,2,\ldots ,n.$

Firstly, we assumme that $n$ is positive odd integer $(n=2p+1,p\in 
\mathbb{N}
).$

Denoting \ $j$th column of the inverse matrix $T^{-1}$ by $\sigma _{j~}$and
implementing necessary transformations, we have

\begin{equation*}
T_{j}^{-1}=\left[ 
\begin{array}{c}
\mu _{1}r_{j-1}U_{j-1}\left( \frac{\psi _{1}}{2}\right) \\ 
\mu _{2}r_{j-1}U_{j-1}\left( \frac{\psi _{2}}{2}\right) \\ 
\vdots \\ 
\mu _{n-1}r_{j-1}U_{j-1}\left( \frac{\psi _{n-1}}{2}\right) \\ 
\mu _{n}r_{j-1}U_{j-1}\left( \frac{\psi _{n}}{2}\right)%
\end{array}%
\right]
\end{equation*}%
where

\begin{equation*}
\mu _{k}=\left\{ 
\begin{array}{l}
~~\frac{1}{2n+2}\psi _{\frac{n+1}{2}+k}^{2},1\leq k\leq \frac{n-1}{2} \\ 
~~~~~\frac{2}{n+1}~~~~,~~~~k=\frac{n+1}{2} \\ 
\frac{1}{2n+2}\psi _{\frac{3(n+1)}{2}-k}^{2},\frac{n+3}{2}\leq k\leq n%
\end{array}%
k=1,2,\ldots ,n(n=2p+1,p\in 
\mathbb{N}
).\right.
\end{equation*}

Let

\begin{equation*}
\left( A^{\dag }\right) ^{s}=T\left( J^{\dag }\right)
^{s}T^{-1}=W(s)=(w_{ij}(s))
\end{equation*}%
here $s\in 
\mathbb{N}
~(n=2p+1,p\in 
\mathbb{N}
).~$Hence

\begin{equation}
w_{ij}(s)=\sum_{k=1}^{n}\left( \lambda _{k}^{\dag }\right) ^{s}\mu
_{k}r_{i-1}r_{j-1}U_{i-1}\left( \frac{\psi _{k}}{2}\right) U_{j-1}\left( 
\frac{\psi _{k}}{2}\right)  \label{30}
\end{equation}%
where $i=1,2,\ldots ,n;~j=1,2,\ldots ,n.$

Secondly we assume that $n$ is positive even integer $(n=2p,p\in 
\mathbb{N}
)~[8]$.

Denoting \ $j$th column of the inverse matrix $T^{-1}$ by $\sigma _{j}$ and
implementing necessary transformations, we obtain

\begin{equation*}
T_{j}^{-1}=\left[ 
\begin{array}{c}
\eta _{1}r_{j-1}U_{j-1}\left( \frac{\psi _{1}}{2}\right) \\ 
\eta _{2}r_{j-1}U_{j-1}\left( \frac{\psi _{2}}{2}\right) \\ 
\vdots \\ 
\eta _{n-1}r_{j-1}U_{j-1}\left( \frac{\psi _{n-1}}{2}\right) \\ 
\eta _{n}r_{j-1}U_{j-1}\left( \frac{\psi _{n}}{2}\right)%
\end{array}%
\right]
\end{equation*}%
where

\begin{equation*}
\eta _{k}=\frac{4-\psi _{k}^{2}}{2n+2},k=1,2,\ldots ,n(n=2p,p\in 
\mathbb{N}
).
\end{equation*}

Let

\begin{equation*}
\left( A^{\dag }\right) ^{s}=T\left( J^{\dag }\right)
^{s}T^{-1}=L(s)=(l_{ij}(s))
\end{equation*}%
where $s\in 
\mathbb{Z}
~(n=2p,p\in 
\mathbb{N}
).~$So

\begin{equation}
l_{ij}(s)=\sum_{k=1}^{n}\left( \lambda _{k}^{\dag }\right) ^{s}\eta
_{k}r_{i-1}r_{j-1}U_{i-1}\left( \frac{\psi _{k}}{2}\right) U_{j-1}\left( 
\frac{\psi _{k}}{2}\right)  \label{31}
\end{equation}%
where $i=1,2,\ldots ,n;~j=1,2,\ldots ,n.$

\begin{corollary}
Let%
\begin{equation}
\tilde{A}^{\dag }=\left[ 
\begin{array}{cccccc}
0 &  &  &  & b & a \\ 
&  &  & -b & a & b \\ 
&  & b & a & -b &  \\ 
& -b & a & b &  &  \\ 
{\mathinner{\mkern2mu\raise1pt\hbox{.}\mkern2mu
\raise4pt\hbox{.}\mkern2mu\raise7pt\hbox{.}\mkern1mu}} & a & -b &  &  &  \\ 
{\mathinner{\mkern2mu\raise1pt\hbox{.}\mkern2mu
\raise4pt\hbox{.}\mkern2mu\raise7pt\hbox{.}\mkern1mu}} & {%
\mathinner{\mkern2mu\raise1pt\hbox{.}\mkern2mu
\raise4pt\hbox{.}\mkern2mu\raise7pt\hbox{.}\mkern1mu}} &  &  &  & 0%
\end{array}%
\right]  \label{32}
\end{equation}%
be the anti-tridiagonal matrices, where $0\neq b,a\in 
\mathbb{C}
$.
\end{corollary}

\begin{lemma}
Let \ $0\neq b,a\in 
\mathbb{C}
,~n=2p,~p\in 
\mathbb{N}
$, $A^{\dag }$ and 
\begin{equation}
\tilde{J}^{\dag }=\left[ 
\begin{array}{ccccc}
0 &  &  &  & 1 \\ 
&  &  & 1 &  \\ 
&  & {\mathinner{\mkern2mu\raise1pt\hbox{.}\mkern2mu
\raise4pt\hbox{.}\mkern2mu\raise7pt\hbox{.}\mkern1mu}} &  &  \\ 
& 1 &  &  &  \\ 
1 &  &  &  & 0%
\end{array}%
\right] .  \label{33}
\end{equation}%
Then%
\begin{equation}
\tilde{A}^{\dag }=\tilde{J}^{\dag }A^{\dag }=A^{\dag }\tilde{J}^{\dag }.
\label{34}
\end{equation}
\end{lemma}

\begin{proof}
See [11].
\end{proof}

\begin{lemma}
\bigskip Let $\tilde{A}^{\dag }$ be as in ($32$). Then 
\begin{equation}
\left( \tilde{A}^{\dag }\right) ^{s}=\left\{ 
\begin{array}{l}
~~~~\left( A^{\dag }\right) ^{s},~~s~is~even \\ 
\tilde{J}^{\dag }\left( A^{\dag }\right) ^{s},~~s~is~odd.%
\end{array}%
\right.  \label{35}
\end{equation}
\end{lemma}

\begin{proof}
We will use by induction on $s$. The case $s=1$ is explicit. Assume that the
equality ($35$) is true for$~s>1$. Now let us show the equality ($35$) is
true for $s+1$. By the induction hypothesis we possess%
\begin{equation*}
\left( \tilde{A}^{\dag }\right) ^{s+1}=\left\{ 
\begin{array}{l}
\tilde{A}^{\dag }\tilde{J}^{\dag }\left( A^{\dag }\right) ^{s}~,~~s+1~is~even
\\ 
~~~~\tilde{A}^{\dag }\left( A^{\dag }\right) ^{s},~~s+1~is~odd.%
\end{array}%
\right.
\end{equation*}%
Since $~\tilde{A}^{\dag }=A^{\dag }\tilde{J}^{\dag }$, we get 
\begin{equation*}
\left( \tilde{A}^{\dag }\right) ^{s+1}=\left\{ 
\begin{array}{c}
~~\left( A^{\dag }\right) ^{s+1}~,~~s+1~is~even \\ 
\tilde{J}^{\dag }\left( A^{\dag }\right) ^{s+1},~~s+1~is~odd.%
\end{array}%
\right.
\end{equation*}
\end{proof}

\begin{theorem}
Let $\tilde{A}^{\dag }$ be $n-$square complex anti-tridiagonal matrix in ($%
32 $). If $s$ is odd, then the $s-$th power of \bigskip $\tilde{A}^{\dag }$
is%
\begin{equation}
\chi _{n-i+1,j}^{s}=\sum_{k=1}^{n}\left( \lambda _{k}^{\dag }\right)
^{s}\eta _{k}r_{n-i}r_{j-1}U_{n-i}\left( \frac{\psi _{k}}{2}\right)
U_{j-1}\left( \frac{\psi _{k}}{2}\right)  \label{36}
\end{equation}%
and if $s$ is even, then the $s-$th power of $\tilde{A}^{\dag }$ is%
\begin{equation}
\chi _{i,j}^{s}=\sum_{k=1}^{n}\left( \lambda _{k}^{\dag }\right) ^{s}\eta
_{k}r_{i-1}r_{j-1}U_{i-1}\left( \frac{\psi _{k}}{2}\right) U_{j-1}\left( 
\frac{\psi _{k}}{2}\right)  \label{37}
\end{equation}%
for \bigskip\ $i=1,2,\ldots ,n;~j=1,2,\ldots ,n.$
\end{theorem}

\begin{proof}
Let $\left( \tilde{A}^{\dag }\right) ^{s}=\left( \chi _{ij}^{s}\right) $ and 
$L=A^{\dag }.$ We obtain the eigenvalues of $A^{\dag }$ as $\left( 14\right) 
$ and the entries of the matrix $L$ as $\left( 31\right) $. Let $s$ be odd
interger. If we multiply the equality $\left( 31\right) $ by $\tilde{J}%
^{\dag }$ from left side, then we possess%
\begin{equation*}
\left( \tilde{J}^{\dag }A^{\dag }\right) _{i,k}=\underset{r=1}{\overset{n}{%
\sum }}\left( \tilde{J}^{\dag }\right) _{i,r}l_{r,k}(s)=l_{n-i+1,k}\left(
s\right) ;~k=1,\ldots ,n.
\end{equation*}%
Therefore we get 
\begin{eqnarray*}
\chi _{n-i+1,j}^{s} &=&l_{n-i+1,k}\left( s\right) \\
&=&\sum_{k=1}^{n}\left( \lambda _{k}^{\dag }\right) ^{s}\eta
_{k}r_{n-i}r_{j-1}U_{n-i}\left( \frac{\psi _{k}}{2}\right) U_{j-1}\left( 
\frac{\psi _{k}}{2}\right) ;~i,j=\overline{1,n}.
\end{eqnarray*}%
Let $s$ be even, then the equality $\left( 31\right) $ is valid by the
equality $\left( 35\right) $.
\end{proof}

\section{\textbf{Numerical \ examples}}

\begin{example}
If $s=3$ and $~n=3.$ Then%
\begin{equation*}
A=\left[ 
\begin{array}{ccc}
a & 2b & 0 \\ 
b & a & 2b \\ 
0 & b & a%
\end{array}%
\right] .
\end{equation*}%
We have%
\begin{equation*}
J=diag(\lambda _{1},\lambda _{2},\lambda _{3})=diag(a+2b,a,a-2b)
\end{equation*}%
and%
\begin{equation*}
A^{3}=(u_{ij}(s))=(u_{ij}(3))=\left[ 
\begin{array}{ccc}
u_{11} & u_{12} & u_{13} \\ 
u_{21} & u_{22} & u_{23} \\ 
u_{31} & u_{32} & u_{33}%
\end{array}%
\right] =\left[ 
\begin{array}{rrr}
x & 2y & 4z \\ 
y & q & 2y \\ 
z & y & x%
\end{array}%
\right] ;
\end{equation*}%
\begin{equation*}
x=a^{3}+6ab^{2},~~y=3a^{2}b+4b^{3},~~z=3ab^{2},~~q=a^{3}+12ab^{2}.
\end{equation*}
\end{example}

\begin{example}
$s=-3,~~n=4,~a=1$ and $b=2$. Then, we get%
\begin{eqnarray*}
J &=&diag(\lambda _{1},\lambda _{2},\lambda _{3},\lambda
_{4})=diag(a+2b,a+b,a-b,a-2b) \\
&=&diag\left( 5,3,-1,-3\right)
\end{eqnarray*}%
and%
\begin{equation*}
A^{-3}=(u_{ij}(-3))=\frac{1}{1000}\left[ 
\begin{array}{rrrr}
-326 & 361 & 311 & -676 \\ 
180 & -170 & -158 & 311 \\ 
156 & -158 & -170 & 361 \\ 
-169 & 156 & 180 & -326%
\end{array}%
\right] .
\end{equation*}
\end{example}

\begin{example}
If$~s=4,~n=3$. Then 
\begin{equation*}
A^{\dag }=\left[ 
\begin{array}{rrr}
a & b & 0 \\ 
b & a & -b \\ 
0 & -b & a%
\end{array}%
\right] .
\end{equation*}%
We achieve%
\begin{equation*}
J^{\dag }=diag(\lambda _{1},\lambda _{2},\lambda _{3})=diag(a-b\sqrt[~]{2}%
,a,a+b\sqrt[~]{2})
\end{equation*}%
and%
\begin{equation*}
\left( A^{\dag }\right) ^{4}=w_{ij}(4)=\left[ 
\begin{array}{rrr}
w_{11} & w_{12} & w_{13} \\ 
w_{21} & w_{22} & w_{23} \\ 
w_{31} & w_{32} & w_{33}%
\end{array}%
\right] =\left[ 
\begin{array}{rrr}
x^{\dag } & y^{\dag } & z^{\dag } \\ 
y^{\dag } & q^{\dag } & -y^{\dag } \\ 
z^{\dag } & -y^{\dag } & x^{\dag }%
\end{array}%
\right] ;
\end{equation*}%
\begin{equation*}
x^{\dag }=a^{4}+6a^{2}b^{2}+2b^{4},~~y^{\dag }=4a^{3}b+8ab^{3},~~z^{\dag
}=-6a^{2}b^{2}-2b^{4},~~q^{\dag }=a^{4}+6a^{2}b^{2}+2b^{4}.
\end{equation*}
\end{example}

\begin{example}
If $~n=4,~s=4,~a=1~$and $b=4$, then 
\begin{eqnarray*}
J^{\dag } &=&diag(\lambda _{1},\lambda _{2},\lambda _{3},\lambda
_{4})=diag(a-\frac{b}{2}(1+\sqrt[~]{5}),a-\frac{b}{2}(\sqrt[~]{5}-1),a-\frac{%
b}{2}(1-\sqrt[~]{5}),a+\frac{b}{2}(1+\sqrt[~]{5})) \\
&=&diag(-1-2\sqrt[~]{5},3-2\sqrt[~]{5},-1+2\sqrt[~]{5},3+2\sqrt[~]{5}).
\end{eqnarray*}%
Therefore

\begin{equation*}
\left( A^{\dag }\right) ^{4}=l_{ij}(4)=\left[ 
\begin{array}{rrrr}
609 & 528 & -864 & -256 \\ 
528 & 1473 & -784 & -864 \\ 
-864 & -784 & 1473 & 528 \\ 
-256 & -864 & 528 & 609%
\end{array}%
\right] .
\end{equation*}
\end{example}

\begin{example}
If $~n=4,~s=-5,~a=i~~$and $b=1$, then%
\begin{equation*}
J^{\dag }=diag(\lambda _{1},\lambda _{2},\lambda _{3},\lambda _{4})=diag(i-%
\frac{1}{2}(1+\sqrt[~]{5}),i-\frac{1}{2}(\sqrt[~]{5}-1),i-\frac{1}{2}(1-\sqrt%
[~]{5}),i+\frac{1}{2}(1+\sqrt[~]{5})).
\end{equation*}%
So%
\begin{equation*}
\left( A^{\dag }\right) ^{-5}=l_{ij}(-5)=\frac{1}{1000}\left[ 
\begin{array}{rrrr}
296i & 56~ & 192i & 128~ \\ 
56~ & 104i & 72~ & 192i \\ 
192i & 72~ & 104i & 56~ \\ 
128~ & 192i & 56~ & 296i%
\end{array}%
\right] .
\end{equation*}
\end{example}

\section{Complex Factorizations}

The well-known Fibonacci polynomials $F(x)=\left\{ F_{n}(x)\right\}
_{n=1}^{\infty }$ are defined in $[12]$ by the recurrence relation

\begin{equation}
F_{n}(x)=xF_{n-1}(x)+F_{n-2}(x)  \label{38}
\end{equation}%
where $F_{0}(x)=0$ , $F_{1}(x)=1$ and $n\geq 3$.

\begin{theorem}
Let the matrix $A$ be $n$-square matrix as in $(1)$ with $a:=x$ and $b:=%
\mathbf{i}$ \ where $\mathbf{i}=\sqrt{-1}.$ Then%
\begin{equation}
\det (A)=(x^{2}+4)F_{n-1}(x)  \label{39}
\end{equation}%
where $F_{n}$ is $n$th Fibonacci number.

\begin{proof}
Applying Laplace expansion according to the first two and the last two rows
of the matrix $A$, we have%
\begin{equation*}
\det (A)=x^{2}\check{D}_{n-2}+4x\check{D}_{n-3}+4\check{D}_{n-4}
\end{equation*}%
here $~\check{D}_{n}=\det (tridiag_{n}(\mathbf{i},x,\mathbf{i})).$ Since%
\begin{equation*}
\det (tridiag_{n}(\mathbf{i},x,\mathbf{i}))=F_{n+1}(x),
\end{equation*}%
we have%
\begin{eqnarray*}
\det (A) &=&x^{2}F_{n-1}(x)+4xF_{n-2}(x)+4F_{n-3}(x) \\
&=&x^{2}(xF_{n-2}(x)+F_{n-3}(x))+4xF_{n-2}(x)+4F_{n-3}(x) \\
&=&(x^{2}+4)(xF_{n-2}(x)+F_{n-3}(x))=(x^{2}+4)F_{n-1}(x).
\end{eqnarray*}%
So that, the proof is completed.
\end{proof}
\end{theorem}

\begin{corollary}
Let the matrix $A$ be as in $(1)$ with $a:=x$ and $b:=\mathbf{i.}$ Then the
complex factorization of generalized Fibonacci-Pell numbers is the following
form:%
\begin{equation}
F_{n-1}(x)=\frac{1}{x^{2}+4}\dprod\limits_{k=1}^{n}\left( x+2\mathbf{i}\cos
\left( \frac{(k-1)\pi }{n-1}\right) \right) .  \label{40}
\end{equation}
\end{corollary}

\begin{proof}
Since the eigenvalues of the matrix $A$ from $(7)$%
\begin{equation*}
\lambda _{k}=x+2\mathbf{i}\cos \left( \frac{(k-1)\pi }{n-1}\right) ,\
k=1,2,\ldots ,n
\end{equation*}%
the determinant of the matrix $A$ can be expressed as%
\begin{equation*}
\det (A)=\dprod\limits_{k=1}^{n}\left( x+2\mathbf{i}\cos \left( \frac{%
(k-1)\pi }{n-1}\right) \right) .
\end{equation*}%
By considering $(40)$ and Theorem $12$, the complex factorization of
generalized Fibonacci-Pell numbers is achieved.
\end{proof}

\textbf{Acknowledgement. }The authors are partially supported by TUBITAK and
the Office of Sel\c{c}uk University Research Project (BAP).


\begin{thebibliography}{10}
\bibitem[1]{mohamed} M. Elouafi \& A. D. A. Hadj, On the powers and the
inverse of a tridiagonal matrix, Applied Mathematics and Computation, 211,
(2009) 137-141.

\bibitem[2]{Gutiérrez} Jes\'{u}s Guti\'{e}rrez-Guti\'{e}rrez, Powers of
tridiagonal matrices with constant diagonals, Applied Mathematics and
Computation, 206, (2008) 885-891.

\bibitem[3]{Jesús} Jes\'{u}s Guti\'{e}rrez-Guti\'{e}rrez, Binomial
coefficientsand powers oflarge tridiagonal matrices with constant diagonals,
Applied Mathematics and Computation, 219, (2013) 9219-9222.

\bibitem[4]{Jonas} J. Rimas, On computing of arbitrary positive integer
powers for one type of tridiagonal matrices, Applied Mathematics and
Computation, 161, (2005) 1037-1040.

\bibitem[5]{JR} J. Rimas, On computing of arbitrary positive integer powers
for one type of symmetric tridiagonal matrices of odd order-I, Applied
Mathematics and Computation, 171, (2005) 1214-1217.

\bibitem[6]{Jonas R} J. Rimas, On computing of arbitrary positive integer
powers for one type of symmetric tridiagonal matrices of odd order-II,
Applied Mathematics and Computation, 174, (2006) 676-683.

\bibitem[7]{Rimas} J. Rimas, On computing of arbitrary positive integer
powers for one type of symmetric tridiagonal matrices of even order-I,
Applied Mathematics and Computation, 168, (2005) 783-787.

\bibitem[8]{J.R.} J. Rimas, On computing of arbitrary positive integer
powers for one type of symmetric tridiagonal matrices of even order-II,
Applied Mathematics and Computation, 172, (2006) 245-251.

\bibitem[9]{Akbulak} A. \"{O}tele\c{s} \& M. Akbulak, Positive integer
powers of certain complex tridiagonal matrices, Applied Mathematics and
Computation, 219, (2013) 10448-10455.

\bibitem[10]{M. Akbulak} A. \"{O}tele\c{s} \& M. Akbulak, Positive integer
powers of certain complex tridiagonal matrices, Mathematical Sciences
Letters, 2-1 (2013) 63-72.

\bibitem[11]{H Wang} H. Wang, Powers of complex persymmetric antitridiagonal
matrices with constant antidiagonals, Computational Mathematics, 2014 (2014)
1-10.

\bibitem[12]{thomas} Thomas Koshy, Fibonacci and Lucas Numbers with
Applications, John Wiley and Sons, NY, 2001.

\bibitem[13]{Horn} P.Horn, Ch.Johnson, Matrix Analysis, Cambridge University
Press, \textit{Cambridge}, 1986.

\bibitem[14]{mason} J. C. Mason \& D. C. Handscomb, Chebyshev Polynomials,
CRC Press, \textit{Washington}, 2003.

\bibitem[15]{cahill} N. D. Cahill, J. R. D'Erico \& J. P. Spence, Complex
factorizasyon of the Fibonacci and Lucas numbers, Fibonacci Q. 41 (1) (2003)
13-19.
\end{thebibliography}
\end{document}